\documentclass[12pt]{article}

\usepackage{amssymb}
\usepackage{amsmath}
\usepackage{amsthm}
\usepackage{amsfonts}
\usepackage{enumerate}
\usepackage{url}
\usepackage{graphicx}

\newtheorem{theorem}{Theorem}
\newtheorem{lemma}{Lemma}
\newtheorem{conjecture}{Conjecture}

\newtheorem{corollary}{Corollary}

\theoremstyle{definition}
\newtheorem{definition}{Definition}
\newtheorem{example}{Example}
\renewcommand\labelenumi{(\roman{enumi})}
\renewcommand\theenumi\labelenumi
\newcommand{\Qset}{\normalfont \textbf{Q}_{3x+1}}

\begin{document}

\title{Parity sequences of the 3x+1 map on the 2-adic integers and Euclidean embedding}
\author{Olivier Rozier}
\date{}

\maketitle

\begin{abstract}
In this paper, we consider the one-to-one correspondence between a 2-adic integer and its parity sequence under iteration of the so-called ``$3x+1$'' map. First, we prove a new formula for the inverse transform. Next, we briefly review what is known about the induced automorphism and study its dynamics on the 2-adic integers. We find that it is ergodic on many small odd invariant sets, and that it has two odd cycles of period 2 in addition to its two odd fixed points. Finally, a plane embedding is presented, for which we establish affine self-similarity by using functional equations.
\end{abstract}

\section{Introduction}

It is an unsolved problem \cite{Guy, Lag10} to prove that the repeated iteration of the famous ``$3x+1$" map acting on the positive integers and defined by 
\begin{equation}\label{eq:T}
T(n) = \left\{\begin{array}{ll}
  \frac{3n+1}{2} & \mbox{if $n$ is odd,} \\
  \frac{n}{2} & \mbox{otherwise,} 
\end{array}\right.
\end{equation}
always leads to the value 1, whatever the starting integer of the sequence. And it is not even known whether the orbits $\left( n, T(n), T\left( T(n)\right) , \ldots\right) $ are bounded for all $n$, nor is it known if there exists any non-trivial cycle. This problem, whose origin remains unclear (cf. History and Background section in \cite[p. 5]{Lag10}), has received a great variety of names like the $3x+1$ problem, the Collatz conjecture, the Syracuse problem, Hasse's algorithm, \textit{la conjecture tch\`eque}, etc.

Its intrinsic hardness is frequently attributed to the unpredictability of the successive parities of the iterates in most sequences, until 1 is reached \cite{Aki04,Cra78}. Therefore it seems relevant to study the relationship between the integers and their parity vectors (see Definition \ref{def:parvec}), so as to address the question of the existence and nature of some underlying structure. 

\begin{definition}
\label{def:parvec}
For any two positive integers $j$ and $n$, we call $V_j(n)$ the {\em parity vector} of $n$, of length $j$,
\begin{equation}
V_j(n) = \left(  n, T(n), \ldots , T^{j-1}(n) \right)   \mod 2,
\end{equation}
where $T^k$ denotes the $k$-th iterate of $T$.
\end{definition}
This notion was introduced independently by Everett and Terras, and named the \textit{parity vector} by Lagarias \cite{Lag85}.

It was quite easy to state \cite{Eve77,Ter76} that any two integers have the same parity vector of length $j$ if and only if they belong to the same congruence class modulo $2^j$. From this property, we derive that each function $V_j$ sends with a one-to-one correspondence any set of $2^j$ consecutive integers to the set of all parity vectors of length $j$. There is consequently an infinite class of integers producing exactly any finite sequence of parities under iteration of $T$.
 
\section{Two formulae for the inverse transform}
In this part, we may freely extend the definition of the functions $T$ and $V_j$ to the ring $\mathbb{Z}$ of rational integers, as in \cite{Lag85}.

The transformation of an integer $n$ into its parity vector $V_j(n)$ is straightforward by applying the map $T$ repeatedly $j$ times. Conversely, one may use the forthcoming Lemma \ref{lem:inverse1} to obtain all the integers that have any given parity vector. In fact, it is a well-known expression with various formulations \cite{Boh78, Mon04, Ter76} and generalizations \cite{Ber96,Mol78}, further studied by Matthews and Watts in \cite{Mat84}.

\begin{definition}
Let $j$ be a positive integer. We say that a vector $S$ = ($s_0$, \ldots, $s_{j-1}$) of length $j$ is a finite {\em binary sequence} if $s_k = 0$ or $1$ for all $0 \leq k \leq j-1$.
We further define the {\em partial sum} functions $\sigma_k$ applying on $S$ by
\begin{equation} \label{eq:sigma_k}
\sigma_{k}(S) = \sum_{i=0}^{k} s_{i} \quad \text{for each } k \leq j-1.
\end{equation}
\end{definition}

\vspace{10pt}

The above functions $\sigma_k$ are essentially the same as the functions pop$_k$ introduced in \cite{Ber96} and used in a similar way. They frequently appear in various forms within the literature on the $3x+1$ problem.

\begin{lemma}{\em (First formulation of the inverse transform)} \label{lem:inverse1}
Let $S$ be a finite binary sequence $\left( s_0, s_1, \ldots, s_{j-1} \right)$ of length $j$. The set of integers $n$ for which $V_{j}(n) = S$ is given by the congruence class
\begin{equation} \label{eq:inverse1}
n \equiv - \sum_{k=0}^{j-1} s_{k} \, 2^{k} \, 3^{-\sigma_{k}(S)} \pmod{ 2^{j}}.
\end{equation}
\end{lemma}

\begin{proof}
Suppose that $V_{j}(n) = S$. Then equation \eqref{eq:inverse1} follows from the formula
$$ 2^{j} \, T^{j}(n) = 3^{\sigma_{j-1}(S)} \left( n + \sum_{k=0}^{j-1} s_{k} \, 2^{k} \, 3^{-\sigma_{k}(S)} \right),$$
which is easy to state by induction on $j$ (see \cite{Ter76}).
\end{proof}

\begin{example}
By Lemma \ref{lem:inverse1}, the odd integers $n$ leading to sequences where every odd term is followed by exactly two even terms on the first $j$ iterations of the map $T$ are such that
$$n \equiv - \sum_{k=0}^{\left\lfloor\frac{j-1}{3}\right\rfloor} 8^{k} \, 3^{-(k+1)} \equiv - \frac{1}{3} \, \frac{\left( \frac{8}{3} \right)^{\left\lfloor\frac{j+2}{3}\right\rfloor} - 1}{\frac{8}{3} - 1} \equiv \frac{1}{5} \pmod{ 2^{j}}.$$
For the increasing lengths $j=3,6,9,\ldots$, the smallest positive values of $n$ are $5,13,205,\ldots$ respectively.
\end{example}

The discovery of a second formulation of the inverse transform came after studying the particular case of sequences where all terms but one are odd \cite{Roz17} (see also Example \ref{ex:inverse_vector}). It can be stated in different ways and we give a very short proof using a conjugate function. While this function already appears in the literature (e.g., \cite[p. 6]{Mon13} and \cite[p. 26]{Sim16}), the resulting formula in Theorem \ref{th:inverse2} seems new.

\begin{definition} \label{def:U}
Let us consider the function
$$ \begin{array}{rl}
 U: \mathbb Z & \longrightarrow \mathbb Z\\
 n & \longmapsto \left\{\begin{array}{ll}
  \frac{n+1}{2} & \mbox{if $n$ is odd,} \\
  \frac{3n}{2} & \mbox{otherwise.} 
\end{array}\right.
\end{array}$$
The conjugacy relationship
\begin{equation} \label{eq:U_conjug}
U(n+1) = T(n) + 1 
\end{equation}
holds for all $n$. Furthermore, for any binary sequence $S=\left( s_0, s_1, \ldots, s_{j-1} \right)$ of length $j$ and any integer $n$ such that $U^{k}(n) \equiv s_k \pmod{2}$ for each $k$, one has
\begin{equation} \label{eq:U_inverse}
n \equiv - \sum_{k=0}^{j-1} s_{k} \, 2^{k} \, 3^{\sigma_{k}(S)-k-1} \pmod{ 2^{j}}.
\end{equation}
\end{definition}
The latter congruence is derived from the equation
$$ 2^{j} \, U^{j}(n) = 3^{j-\sigma_{j-1}(S)} \left( n + \sum_{k=0}^{j-1} s_{k} \, 2^{k} \, 3^{\sigma_{k}(S)-k-1} \right),$$
which may be easily proved by induction on $j$, exactly as in Lemma \ref{lem:inverse1}.

We are now able to provide a second formulation of the inverse of the  functions $V_j$, which has little difference with the previous one. It turns out to be practical for sequences of $T$ iterations that contain many odd terms, because the corresponding terms in formula \eqref{eq:inverse2} vanish.

\begin{theorem}{\em (Second formulation of the inverse transform)} \label{th:inverse2}
Let $S$ be a finite binary sequence $\left( s_0, s_1, \ldots, s_{j-1} \right)$. The set of integers $n$ for which $V_{j}(n) = S$ is given by the congruence class
\begin{equation} \label{eq:inverse2}
n \equiv -1 - \sum_{k=0}^{j-1} (1-s_{k}) \, 2^{k} \, 3^{-\sigma_{k}(S)} \pmod{ 2^{j}}.
\end{equation}
\end{theorem}

\begin{proof}
Let $n$ be such that $V_{j}(n) = S$, and consider the binary sequence $$S_U = \left( U^{k}(n+1) \mod 2 \right)_{k=0}^{j-1}.$$ The conjugacy \eqref{eq:U_conjug} gives $U^k(n+1) = T^k(n)+1$, so that $S_U = \left( 1-s_0, \ldots, 1-s_{j-1} \right)$ and $\sigma_{k}(S_U) = k+1-\sigma_{k}(S)$ for every $k$.

It suffices to write the inverse formula \eqref{eq:U_inverse} applied to $n+1$,
$$ n + 1 \equiv - \sum_{k=0}^{j-1} (1-s_{k}) \, 2^{k} \, 3^{\sigma_{k}(S_U)-k-1} \pmod{ 2^{j}},$$
to conclude the proof.
\end{proof}

\begin{example} \label{ex:inverse_vector}
Let $j$ be a positive integer. Suppose we want to find the integers $n$ for which the parity vector $V_{j}(n)$ contains exactly once the value 0. Then we can write that
$$ n \equiv -1 - \left( \frac{2}{3}\right)^{k} \pmod{2^j} $$
where $k$ is the only integer lower than $j$ such that $T^{k}(n)$ is even. See \cite[\S 6]{Roz17} for a brief study of those integers in $\mathbb{Z^+}$.
\end{example}

In fact, we obtain from Lemma \ref{lem:inverse1} and Theorem \ref{th:inverse2} infinitely many formulations by considering linear combinations of \eqref{eq:inverse1} and \eqref{eq:inverse2}. For example, a simple addition gives 
$$ 2n+1 \equiv - \sum_{k=0}^{j-1} 2^{k} \, 3^{-\sigma_{k}(V_j(n))} \pmod{ 2^{j}} \quad \text{for all integers $j>0$ and $n$}.$$
On the other hand, subtracting  the second formulation \eqref{eq:inverse2} from the first formulation \eqref{eq:inverse1} yields the non-trivial congruence in Corollary \ref{cor:invariant}.

\begin{corollary}
\label{cor:invariant}
Let $j$ be a positive integer. For any finite binary sequence $S=\left( s_0, s_1, \ldots, s_{j-1} \right)$, 
\begin{equation}\label{eq:invariant}
\sum_{k=0}^{j-1} (-1)^{s_{k}} \, 2^{k} \, 3^{-\sigma_{k}(S)} \equiv -1 \pmod{ 2^{j}}.
\end{equation}
\end{corollary}

\begin{proof}
Let $n$ be a positive integer such that $V_j(n) = S$. Subtracting each side of \eqref{eq:inverse2} from the corresponding side of \eqref{eq:inverse1} gives the desired result, by writing 
$(-1)^{s_k} = (1-s_{k}) - s_k$ for any $k$.
\end{proof}

This corollary can also be proved directly by induction on $j$, then leading to an alternate proof of Theorem \ref{th:inverse2}, which is left to the reader.

\section{Ultrametric extension}
\label{sec:2-adic}
\subsection{The space of 2-adic integers}
Following Hasses's generalization of the $3x+1$ problem, it was suggested \cite{Mat84,Mol78} to extend the definition of the map $T$ to the ring $\mathbb Z_2$ of 2-adic integers, that is, numbers of the form $\sum_{k=0}^{\infty} a_{k} 2^{k}$ with $a_k = 0$ or 1 for all $k$. The standard shorthand notation $(\ldots a_{2}a_{1}a_{0})_{2}$ from right to left\footnote{Some authors prefer to write the 2-adic ``digits'' from left to right.}  may be used for the sake of conciseness, and the parentheses are most often omitted. A periodic expansion is usually indicated by an upper bar. For example, one may write 
$$ (\ldots010101)_2 = \overline{01}_2 = \sum_{k=0}^{\infty}  2^{2k} = -\frac{1}{3}.$$
Recall that all rational numbers with an odd denominator has an eventually periodic  expansion in $\mathbb Z_2$.
 
A metric can be derived from the 2-adic norm
$$ \left|\sum_{k=0}^{\infty} a_{k} 2^{k}\right|_2 = 2^{-l} \quad \text{with } l=\min \left\lbrace k \geq 0: a_k \neq 0 \right\rbrace, \quad \text{and } |0|_2=0.$$
The space $\mathbb Z_2$ is then said to be \textit{ultrametric}, due to the strong triangle inequality $$|x+z|_2 \leq \max\left( |x+y|_2 , |y+z|_2 \right)$$ for all $x$, $y$ and $z$. Therefore it is not Euclidean.

When needed, we apply the usual Haar measure on $\mathbb Z_2$, here noted $\mu$, such that $\mu(\mathbb Z_2)=1$, and refer to it as the 2-adic measure.

The function $T$ remains well-defined on $\mathbb Z_2$, where it is known to be continuous and measure-preserving \cite{Mat84}. As was observed many times \cite{Aki04, Lag90, Mon04}, iterating $T$ on $\mathbb Z_2$ leads to a much greater variety of behaviors, due to its ergodic and strongly mixing dynamics \cite{Mat84} and interesting properties thus arise.

\subsection{Parity sequences}
\label{sub:parseq}
Let first introduce the notion of \textit{parity sequence}.

\begin{definition} \label{def:V_inf}
For every 2-adic integer $x$, the infinite binary sequence 
\begin{equation} \label{eq:V_inf}
V_{\infty}(x) = \left(  x, T(x), T^{2}(x), \ldots \right)   \mod 2
\end{equation}
is called the {\em parity sequence} of $x$.
\end{definition}

It is remarkable, as mentioned in \cite{Lag85}, that the $V_{\infty}$ function is a one-to-one and onto transform from $\mathbb Z_2$ to $\{0, 1\}^{\infty}$. Every infinite binary sequence is the parity sequence, via $T$ iteration, of exactly one 2-adic integer. As a consequence, there exist 2-adic cycles of every period. A complete list of the 23 cycles of period at most 6 is given in \cite{Lag90}. Since eventually periodic sequences have density zero in $\{0, 1\}^{\infty}$, we infer that almost all orbits in $\mathbb Z_2$ do not contain a cycle.

From Lemma \ref{lem:inverse1} and Theorem \ref{th:inverse2}, one immediately derives two formulae to express the inverse transform $V_{\infty}^{-1}$.

\begin{corollary}
\label{cor:inverse_2adic}
Let $S$ be an infinite binary sequence $\left( s_0, s_1, s_2, \ldots \right)$. The 2-adic integer $x$ such that $V_{\infty}(x) = S$ is given by any of the 2-adically convergent expansions
\begin{equation} \label{eq:inverse_2adic1}
x = - \sum_{k=0}^{\infty} s_{k} \, 2^{k} \, 3^{-\sigma_{k}(S)}
\end{equation}
and
\begin{equation} \label{eq:inverse_2adic2}
x = -1 - \sum_{k=0}^{\infty} (1 -s_{k}) \, 2^{k} \, 3^{-\sigma_{k}(S)}.
\end{equation}
where $\sigma_{k}$ denotes the partial sum function $\sigma_{k}(S)=\sum_{i=0}^{k}s_{i}$.\end{corollary}

\begin{example}\label{ex:inverse_2adic}
Consider the binary sequence $S = \left( s_0, s_1, s_2, \ldots \right)$ where $s_k = 1$ for all $k$. Applying the inverse transform \eqref{eq:inverse_2adic2}, we get $V_{\infty}^{-1}(S) = -1 = \ldots111111_{2}$.
\end{example}

The question whether the inverse formula \eqref{eq:inverse_2adic1} leads to a convergent series when evaluated in the set of real numbers has been investigated in various papers \cite{Edg12,Lop09,Mat84}. Note that the sum of the series is negative or zero, when it exists. Interestingly, both series on the right hand side of \eqref{eq:inverse_2adic1} and \eqref{eq:inverse_2adic2} are expected to be divergent (resp. convergent) for the parity sequences of positive (resp. negative) rational integers. 

One may also remark that the equation \eqref{eq:inverse_2adic1} is quite similar to the expression of the real function $\theta$ mentioned at the end of Coquet's paper \cite[\S 7]{Coq83}, which is convergent and turns out to be fractal.

We do not further discuss this issue in the present paper.

\subsection{Automorphism}
It is convenient to encode every parity sequence as a 2-adic integer, so as to give it a rational value when it is eventually periodic, as done by Lagarias in \cite{Lag85}.  This yields an automorphism in $\mathbb Z_2$.

\begin{definition} \label{def:Q}
 Let $Q$ denote the  function 
$$ \begin{array}{rl}Q:  \mathbb Z_2 & \longrightarrow \mathbb Z_2\\
 x & \longmapsto \sum_{k=0}^{\infty} s_k 2^{k} \end{array}$$
 where $( s_0, s_1, s_2, \ldots )$ is the parity sequence of $x$, as defined by \eqref{eq:V_inf}.
\end{definition}

The function $Q$ is a one-to-one and onto morphism \cite{Ber96,Lag85}. It is also non-expanding\footnote{In \cite{Ber96}, the functions having the property \eqref{eq:Q_sol1} are said to be \textit{solenoidal}.} with respect to the 2-adic norm, since it satisfies the 1-Lipschitz condition 
$$|Q(x) - Q(y)|_{2} \leq |x - y|_{2} \quad \text{for all $x$ and $y$,}$$
or equivalently,
\begin{equation}\label{eq:Q_sol1}
x \equiv y \pmod{2^n} \Longrightarrow  Q(x) \equiv Q(y) \pmod{2^n}.
\end{equation}
The fact that $Q$ is one-to-one further implies (see \cite{Ana06}) the reciprocal
\begin{equation}\label{eq:Q_sol2}
x \equiv y \pmod{2^n} \Longleftrightarrow  Q(x) \equiv Q(y) \pmod{2^n},
\end{equation}
which makes $Q$ a 2-adic isometry \cite{Ber96}.

For convenience purposes, we prefer to use the simple notation $Q$, as in \cite{Aki04}, rather than the original notation $Q_{\infty}$. Its inverse $Q^{-1}$, called the $3x+1$ \textit{conjugacy map} and denoted by $\Phi$ in \cite{Ber96, Mon04, Mon13}, is known \cite{Ber94} to conjugate the map $T$ with the shift map $\mathcal{S}$ whose definition follows. For the sake of clarity, we rephrase all known and conjectured properties related to $\Phi$ in terms of the function $Q$.
\begin{definition}
Let the {\em shift map} $\mathcal{S}$ denote the function 
$$ \begin{array}{rl}\mathcal{S}:  \mathbb Z_2 & \longrightarrow \mathbb Z_2\\
 x & \longmapsto \left\{\begin{array}{ll}
  \frac{x-1}{2} & \mbox{for x odd,} \\
  \frac{x}{2} & \mbox{otherwise.} 
\end{array}\right.
\end{array}$$
The conjugacy
\begin{equation} \label{eq:TS_conjug}
T = Q^{-1} \circ \mathcal{S} \circ Q
\end{equation}
holds.
\end{definition}
 
In the context of the original $3x+1$ problem, it is crucial to determine which values of $Q$ are rational. Indeed, we have the well-known statements below holding for all 2-adic integers $x$:
\begin{itemize}
\item The orbit $\left(  x, T(x), T^{2}(x), \ldots \right) $ is eventually periodic if and only if $Q(x)$ is rational \cite{Aki04}. This is an immediate consequence of the conjugacy \eqref{eq:TS_conjug} as noted by Monks \cite{Mon04}.
\item If $ Q(x) \text{ is rational}$, then $ x \text{ is rational}$  \cite{Ber94} (see also \cite{Aki04}, Theorem 5).
\end{itemize}
It results that all the cycles within the dynamics of $T$ on $\mathbb Z_2$ are rational \cite{Lag90}. Lagarias' Periodicity Conjecture \cite{Lag85} asserts that the reciprocal of the second statement above also holds. This would imply that every rational point in $\mathbb Z_2$ is preperiodic.

\begin{conjecture} {\em (Periodicity Conjecture)} \label{conj:periodicity}
For any 2-adic integer $x$, $Q(x)$ is rational if and only if $x$ is rational.
\end{conjecture}

 Furthermore, the function $Q$ allows to formulate differently the $3x+1$ problem \cite{Ber94,Lag85}.
\begin{conjecture} {\em ($3x+1$ problem)} \label{conj:Q_3x+1}
$$ Q\left( \mathbb Z^{+}\right) \subset \frac{1}{3} \mathbb Z , \quad \text{ or equivalently, } \quad \mathbb Z^{+} \subset Q^{-1}\left( \frac{1}{3} \mathbb Z \right).$$
\end{conjecture}
It asserts that every positive integer has an eventually periodic parity sequence of period 2, ending with an infinite alternation of 0 and 1 (the case of a constant parity is excluded), which only occurs when some iterate reaches the trivial cycle (1,2). Note that the reverse inclusion in Conjecture \ref{conj:Q_3x+1} does not hold, since $Q^{-1}(1)=-1/3$, by formula \eqref{eq:inverse_2adic1}.

\subsection{Functional equations}
\label{sub:feq}
The semiring $\mathbb N$ of natural integers is completely generated by all finite compositions of the functions $x \longmapsto 2x$ and $x\longmapsto 2x+1$ starting from 0, thus reversing the action of the shift map $\mathcal{S}$. Therefore it is tempting to search for functional equations that express $Q(2x)$ and $Q(2x+1)$ from $Q(x)$. Such equations exist for any $x$ in $\mathbb Z_{2}$ or in a subset of $\mathbb Z_{2}$, as shown in Theorem \ref{th:Q_feq}. It turns out that equation \eqref{eq:Q_feq_2xp1} is a sort of 2-adic extension of previous results by Andaloro \cite{And00} and Garner \cite{Gar85}. We also establish similar equations for the inverse transform $Q^{-1}$ (see \cite{Edg12} for a generalization).
\begin{theorem} \label{th:Q_feq}
The functions $Q^{-1}$ and $Q$ are solution to the functional equations
\begin{align}
 Q^{-1}(2 x) &= 2 Q^{-1}(x),  \label{eq:Qinv_feq_2x} \\
 Q^{-1}(2 x + 1) &= \frac{2 Q^{-1}(x) - 1}{3}, \label{eq:Qinv_feq_2xp1} \\
 Q(2 x) &= 2 Q(x)  \label{eq:Q_feq_2x} 
\end{align}
 for all 2-adic integers $x$. Moreover 
\begin{equation}
 Q(2 x + 1) = 2 Q(x) - 2^k + 1  \label{eq:Q_feq_2xp1} 
\end{equation}
for $x \equiv -1 - (-2)^{k-2}  \pmod{2^{k}}$ and $k \geq 2$.
\end{theorem}

\begin{proof}[Proof of \eqref{eq:Qinv_feq_2x} and \eqref{eq:Qinv_feq_2xp1}]
First, one may rewrite equation \eqref{eq:TS_conjug} as
\begin{equation}\label{eq:TS_conjug2}
 T \circ Q^{-1} = Q^{-1} \circ \mathcal{S}. 
\end{equation}
Take a 2-adic integer $x$. Putting together \eqref{eq:TS_conjug2} with the fact that $$\mathcal{S}(2x+1) = \mathcal{S}(2x) = x$$ and that $$Q^{-1}(y) \equiv y \pmod{2} \quad \text{for all }y,$$  we obtain
$$Q^{-1}(x) = Q^{-1} \circ \mathcal{S}(2x) = T \circ Q^{-1}(2x) = \frac{Q^{-1}(2x)}{2}$$
and
$$Q^{-1}(x) = Q^{-1} \circ \mathcal{S}(2x+1)  = T \circ Q^{-1}(2x+1) = \frac{3 Q^{-1}(2x+1) + 1}{2}.$$
\end{proof}

\begin{proof}[Proof of \eqref{eq:Q_feq_2x}]
Replacing $x$ by $Q(x)$ in \eqref{eq:Qinv_feq_2x} gives $Q^{-1}(2 Q(x)) = 2 x$, leading to $2Q(x) = Q(2x)$.
\end{proof}

\begin{proof}[Proof of \eqref{eq:Q_feq_2xp1}]
Let $k \geq 2$ and let $x,y$ be 2-adic integers such that $x = -1 - (-2)^{k-2} +2^{k}y$. 
Starting from $x$ and applying repeatedly the map $T$, it is easily seen that the first $k-3$ iterates are odd, while the next one is even: $T^{k-2}(x) = -1 - (-3)^{k-2} + 3^{k-2} (4y) \equiv 2 \pmod{4}$. Setting $T^{k-2}(x) = 2+4z$, we get $T^{k-1}(x) = 1+2z$ and $T^{k}(x) = 2+3z$. Since $x$ has the parity vector $V_k(x) = (1, 1, \ldots, 1, 0, 1)$, one may write 
$$ Q(x) = 1 + 2 + \ldots + 2^{k-3} + 2^{k-1} + 2^{k} Q( 2+3z ) $$
for $k \geq 3$. In the case $k=2$, the above expression simplifies to $Q(x) = 2 + 4Q(2+3z) $.

On the other hand, starting from $2x+1$ and applying $k-1$ times the map $T$, we get after $(k-2)$ odd iterates the even value $T^{k-1}(2x+1) = -1 + (-3)^{k-1} + 3^{k-1} (4y) = 3T^{k-2}(x) +2 = 8+12z$. The next two iterates are $T^{k}(2x+1) = 4+6z$ and $T^{k+1}(2x+1) = 2+3z$.  It follows that $2x+1$ has the parity vector $V_{k+1}(2x+1) = (1, 1, \ldots, 1, 0, 0)$, so that 
$$ Q(2x+1) = 1 + 2 + \ldots + 2^{k-2} + 2^{k+1} Q(2+3z)  = 2Q(x) -2^{k} + 1.$$
\end{proof}

One may ask whether Theorem \ref{th:Q_feq} provides a general algorithm to calculate in a finite number of steps the exact value of the function $Q$ applied to an arbitrary positive integer. Unfortunately, the answer appears to be negative, since equation \eqref{eq:Q_feq_2xp1} only applies to a subset of $\mathbb N$ of density $2^{-2}+2^{-3}+2^{-4}+\ldots = 1/2$.

The case $k=2$  can be restated as
\begin{equation}
Q(4 x + 1) = 4 Q(x) - 3  \quad \text{for }x \equiv 1  \pmod{2} \label{eq:Q_feq_4xp1} 
\end{equation}
by replacing $x$ by $2x$ in \eqref{eq:Q_feq_2xp1} and using \eqref{eq:Q_feq_2x}. One may further combine \eqref{eq:Q_feq_2xp1} and \eqref{eq:Q_feq_4xp1} to produce the functional equations
$$ Q(8x+5) = 4Q(2x+1) - 3 = 8Q(x) - 2^{k+2} + 1 $$
for $x \equiv -1 - (-2)^{k-2}  \pmod{2^{k}}$ and $k \geq 2$.

The function $Q$ satisfies many other functional equations that are not combinations of \eqref{eq:Q_feq_2x} and \eqref{eq:Q_feq_2xp1} like
\begin{equation} \label{eq:Q_feq_3xp1}
Q(3x+1) = Q(x) - 1\quad \text{for }x \equiv 1  \pmod{2}. 
\end{equation}
Such equations are always related to the phenomenon of coalescence within the dynamics of $T$. For example, the equation \eqref{eq:Q_feq_3xp1} derives directly from the equality $T(3x+1) = T(x)$ for $x \equiv 1 \pmod{2}$; see \cite{And00,Gar85} for other examples of generic coalescences.

\subsection{Ergodicity}
\label{sub:erg}
The ergodic dynamics of the $3x+1$ map $T$ on $\mathbb{Z}_2$ is quite well understood, and paradoxically, it does not provide any indication on the validity of the $3x+1$ Conjecture, as is discussed in \cite{Aki04}.

Nevertheless, in view of the Periodicity Conjecture, it could be helpful to better specify the dynamics of $Q$, which appears to be more complicated.

In what follows, we refer to \cite{Ana06} for the ergodicity of a measure-preserving function on the 2-adic integers.

Since $Q$ is isometric, it induces in the finite set $\mathbb{Z}/2^n\mathbb{Z}$ a permutation $Q_n$ whose behavior is easier to study.
\begin{definition}
For all integers $n \geq 0$, let $Q_n$ denote the function
$$ \begin{array}{rl}Q_n:  \mathbb{Z}/2^n\mathbb{Z} & \longrightarrow \mathbb{Z}/2^n\mathbb{Z}\\
 x & \longmapsto Q(x) \mod 2^n. \end{array}$$
\end{definition}

In \cite{Lag85}, Lagarias showed that the order of $Q_n$ is always a power of 2, and the following theorem was finally stated in \cite{Ber96}.

\begin{theorem} {\em (Bernstein, Lagarias)} \label{th:Qn_cycles}
For every positive integer $n$, the length of any cycle in $Q_n$ is a power of 2. Moreover, $Q_n$ is a permutation of order $2^{n-4}$ for $n \geq 6$.
\end{theorem}

When lifting from $\mathbb{Z}/2^n\mathbb{Z}$ to $\mathbb{Z}/2^{n+1}\mathbb{Z}$, it is known that any cycle of $Q_n$  either splits into two cycles whose period is unchanged, or undergoes a period-doubling.
\begin{definition} \label{def:edp_cycle}
Let $m \geq k \geq 0$ and let $C= (c_{1}, \ldots, c_{2^k})$ be a cycle of the permutation $Q_m$ of length $2^k$. We say that $C$ has an \textit{ever-doubling period} if, for all $n \geq m$, the elements $c_1, \ldots, c_{2^k}$ of $C$ are all included in a single cycle of $Q_n$ of length $2^{n-m+k}$.
\end{definition}

The second part of Theorem \ref{th:Qn_cycles} is based on the fact that the cycle (5,17) of $Q_5$ has an ever-doubling period (see \cite{Ber96}).

Now we can use this result to study the dynamics of $Q$ and $Q^{-1}$ on the topological space $\mathbb Z_{2}$. To this aim, we need the notion of 2-adic {\em ball}.

\begin{definition} \label{def:ball}
For any $y \in \mathbb Z_{2}$ and $r \geq 0$, let $B(y,r)$ denote the (closed) {\em ball} $$B(y,r)=\left\lbrace x \in \mathbb Z_{2}: |x-y|_{2} \leq r \right\rbrace $$  with center $y$ and radius $r$. Equivalently, one has $$B(y,2^{-k}) = \left\lbrace x \in \mathbb Z_{2}: x \equiv y \pmod{2^{k}} \right\rbrace$$ for every integer $k \geq 0$, and its 2-adic measure is given by its radius: $$\mu\left( B(y,2^{-k})\right) = 2^{-k}.$$
\end{definition}

Recall that the function $Q$ is measure-preserving \cite{Ber96,Lag85} and, unlike the map $T$, is not ergodic on $\mathbb{Z}_2$, since it preserves the parity.

One may at first observe that all forward and backward orbits remain close to the initial point, and that the 2-adic distance is even smaller when the number of iterations is highly divisible by 2. This fact is illustrated in the table below that gives some of the iterates of the 2-adic integer $\overline{0110}1_2 = 1/5$.
{\renewcommand{\arraystretch}{1.5}
\begin{center}
\begin{tabular}{c|rc|rc}
$j$ & $Q^{j}\left( \frac{1}{5}\right)$ & $\left| Q^{j}\left( \frac{1}{5}\right) - \frac{1}{5}\right|_{2}$ &  $Q^{-j}\left( \frac{1}{5}\right)$ & $\left| Q^{-j}\left( \frac{1}{5}\right) - \frac{1}{5}\right|_{2}$\\
 \hline
 1 & $-\frac{1}{7}=\overline{001}_2$ & $2^{-2}$ & $\frac{13}{21}=\overline{001100}1_2$ & $2^{-2}$ \\
 2 & $\frac{17}{5}=\overline{0011}101_2$ & $2^{-4}$ & $-\frac{1}{11}=\overline{0001011101}_2$ & $2^{-4}$ \\
 3 & $\frac{1863}{31}=\ldots1001_2$ & $2^{-2}$ & $\frac{373}{781}=\ldots1001_2$ & $2^{-2}$ \\
 4 & $\ldots00001101_2$ & $2^{-6}$ & $\ldots10001101_2$ & $2^{-6}$ \\
 \end{tabular}
\end{center}
The previous observation is due to the congruences \eqref{eq:Q2} and \eqref{eq:Q2k}, which follow from Theorem \ref{th:Qn_cycles}.
\begin{corollary} \label{cor:Q_iter}
For all 2-adic integers $x$ and all $k \geq 2$,
\begin{align}
 Q^{2}(x) &\equiv x \pmod{2^4}, \label{eq:Q2}\\
 Q^{2^k}(x) &\equiv x \pmod{2^{k+4}} \label{eq:Q2k},
\end{align}
or equivalently, $ Q^{2}(x) \in B(x,2^{-4})$ and $Q^{2^k}(x) \in B(x,2^{-k-4})$.
\end{corollary}

\begin{proof}
Take an integer $k \geq 2$. Theorem \ref{th:Qn_cycles} implies that the length of every cycle of $Q_{k+4}$ divides $2^k$, from which we infer the congruence \eqref{eq:Q2k}.

Likewise, the permutation $Q_4$ has ten fixed points and three cycles of length 2, which are (1,5), (2,10), and (9,13). Hence, the equation \eqref{eq:Q2}.
\end{proof}

Consequently, whatever the 2-adic integer $x$, its forward and backward orbits under iteration of  $Q$ have elements arbitrarily close to $x$.

\begin{corollary}
For all 2-adic integers $x$,
$$ \lim_{k\rightarrow \infty} Q^{2^k}(x) = \lim_{k\rightarrow \infty} Q^{-2^k}(x) = x.$$
\end{corollary}

Though the dynamics of $Q$ is not truly ergodic on $\mathbb Z_2$, this may occur on some invariant subsets.

Let us recall a known criterion for the ergodicity of non-expanding\footnote{In \cite{Ana06}, the term \textit{compatible} is used instead of non-expanding for the same meaning.} functions, by Anashin (Proposition 4.1 in \cite{Ana06}; see also \cite{Ana14}).
\begin{theorem} {\em (Anashin)} \label{th:ergodic}
A non-expanding function $F: \mathbb Z_2 \rightarrow \mathbb Z_2$ is ergodic if and only if $F$ induces modulo $2^n$ a permutation with a single cycle for all positive integers $n$.
\end{theorem}

The next theorem shows that $Q$ is ergodic in a neighborhood of each cycle of $Q_m$ having an ever-doubling period, for any $m \geq 0$.

\begin{theorem} \label{th:Q_ergodic}
Let $Q_m$ denote the permutation induced by $Q$ in $\mathbb{Z}/2^m\mathbb{Z}$. For all $m \geq k \geq 0$ and all cycles $C= (c_{1}, \ldots, c_{2^k})$ of $Q_m$ having an ever-doubling period, the restriction of $Q$ to $B(c_1, 2^{-m}) \cup \ldots \cup B(c_{2^k}, 2^{-m})$ is ergodic.
\end{theorem}
\begin{proof}
Let $n \geq m$. Put $K=B(c_1, 2^{-m}) \cup \ldots \cup B(c_{2^k}, 2^{-m})$ and $K_n=K \mod 2^n$. Since $Q$ is isometric, the sets $K$ and $K_n$ are left invariant by $Q$ and $Q_n$ respectively.

Let $C_n$ be the cycle of the permutation $Q_n$ that contains all the elements of $C$. Its length is equal to $2^{n-m+k}$. Moreover, it is included in $K_n$ whose cardinality is equal to $2^{n-m+k}$. Therefore the restriction of $Q_n$ to the set $K_n$ is a permutation with a single cycle $C_n$. 

From Theorem \ref{th:ergodic}, we deduce that the restriction of $Q$ to the set $K$ is ergodic. For completeness, it is not difficult to find a suitable bijection $F: \mathbb Z_2 \rightarrow K$ for which the conjugate function $F^{-1} \circ Q \circ F$ acting on $\mathbb Z_2$ is non-expanding and ergodic. For example, one may use the function
$$ \begin{array}{rl}F:  \mathbb Z_2 & \longrightarrow K\\
 x & \longmapsto c_{i+1} + 2^{m-k} (x-i), \quad \text{where }i = x \mod 2^k,
 \end{array}$$
 which is one-to-one and onto.
\end{proof}

\begin{definition} \label{def:ergodic}
Whenever $Q$ is ergodic on a (closed) invariant set with positive measure, we call it an {\em ergodic set}. Moreover, we call {\em ergodic domain} the union of all the ergodic sets.
\end{definition}
The fact that $Q$ is bijective further implies that $Q^{-1}$, namely, the $3x+1$ conjugacy map, has the same ergodic domain as $Q$.

\begin{table}
\caption{Odd cycles $C$ of $Q_m$ of length $2^k$ and having an ever-doubling period for $0 \leq m-k \leq 6$. The last column gives the 2-adic measure, equal to $2^{k-m}$, of their $\omega$-limit sets for the function $Q$.}
\label{tab:ergodic}
\begin{center}
\begin{tabular}{|c|c|c|c|}
\hline
$m$ & $k$ & $C$ & $\mu\left( \omega(C) \right)$  \\
\hline
\hline
5 & 1 & (5,17) & $2^{-4}$ \\
\hline
6 & 2 & (9, 29, 25, 13) & $2^{-4}$ \\
\hline
6 & 2 & (41, 61, 57, 45) & $2^{-4}$ \\
\hline
8 & 2 & (27, 251, 219, 59) & $2^{-6}$ \\
\hline
8 & 2 & (91, 187, 155, 123) & $2^{-6}$ \\
\hline
\end{tabular}
\end{center}
\end{table}

Let us point out that every ergodic set in Theorem \ref{th:Q_ergodic} is closed, hence it is the $\omega$-limit set in $\mathbb Z_{2}$ of any point of the associated cycle $C$. For convenience, we write $\omega(C)$ to refer to this set.

In order to identify the cycles having an ever-doubling period, it is convenient to use the following criterion (see Theorem 3.1 in \cite{Ber96}) whose original formulation and vocabulary have been significantly modified.

\begin{theorem} {\em (Bernstein, Lagarias)} \label{th:cycle_edp}
Let $m \geq k \geq 2$ and let $C$ be a cycle of $Q_m$ of length $2^k$. If $C$ is part of a cycle of $Q_{m+2}$ of length $2^{k+2}$, then $C$ has an ever-doubling period.
\end{theorem}

\begin{table}
\caption{Numbers $N_k$ of odd ergodic sets of 2-adic measure $2^{-k}$, $k \leq 16$.}
\label{tab:nb_ergodic}
\begin{center}
\begin{tabular}{cc}
\begin{tabular}{|c|c|c|}
\hline
$k$ & $N_k$ & $N_k \times 2^{-k}$  \\
\hline
\hline
1 & 0 & 0.000 \\
\hline
2 & 0 & 0.000 \\
\hline
3 & 0 & 0.000 \\
\hline
4 & 3 & 0.187 \\
\hline
5 & 0 & 0.000 \\
\hline
6 & 2 & 0.031 \\
\hline
7 & 10 & 0.078 \\
\hline
8 & 11 & 0.042 \\
\hline
\end{tabular}

\begin{tabular}{|c|c|c|}
\hline
$k$ & $N_k$ & $N_k \times 2^{-k}$  \\
\hline
\hline
9 & 11 & 0.021 \\
\hline
10 & 29 & 0.028 \\
\hline
11 & 54 & 0.026 \\
\hline
12 & 91 & 0.022 \\
\hline
13 & 118 & 0.014 \\
\hline
14 & 213 & 0.013 \\
\hline
15 & 282 & 0.008 \\
\hline
16 & 436 & 0.006 \\
\hline
\end{tabular}

\end{tabular}
\end{center}
\end{table}

After some straightforward numerical computations, we find, by applying Theorems \ref{th:Q_ergodic} and \ref{th:cycle_edp} above, that $Q$ is ergodic on the $\omega$-limit set of every cycle in Table \ref{tab:ergodic}. Though it is not the case for any other odd set of 2-adic measure at least $2^{-6}$, there are many smaller odd ergodic sets. In Table \ref{tab:nb_ergodic}, we provide their respective numbers when sorted by size (see also Table 2.2 in \cite{Ber96}). We obtain that the total measure of the odd ergodic sets exceeds 0.48, whereas it is trivially upper bounded by $\mu \left(B(1, 1/2) \right) = 1/2$.

For each odd ergodic set of measure $2^{-k}$ and each $m \geq 1$, we easily get, by applying the autoconjugacy \eqref{eq:Q_feq_2x} repeatedly $m$ times, an even ergodic set of measure $2^{-k-m}$. It yields that the ergodic domain has a measure greater than 0.96.

\begin{conjecture} {\em (Ergodicity Conjecture)}
The ergodic domain of $Q$ has full 2-adic measure.
\end{conjecture}

We expect this conjecture to be closely related to the distribution of periodic orbits, about which little is known.

\subsection{Cycles}
\label{sub:cycles}

The search of the periodic points of the function $Q$ is far from trivial due to the fact that it is nowhere differentiable, as proved by M\"uller in \cite{Mul91} (see also \cite{Ber94} for a short proof).

Hereafter, we call \emph{$Q$-cycle} (resp. \emph{$T$-cycle}) a periodic  orbit of the function $Q$ (resp. $T$). Unlike $T$-cycles, it is not known whether the $Q$-cycles are all rational. Note that the set of even $Q$-cycles may be easily deduced from the odd ones by using  the functional equation \eqref{eq:Q_feq_2x}, and by adding the fixed point $0$. 

As a consequence of Theorem \ref{th:Qn_cycles} (\S\ref{sub:erg}), the period of any $Q$-cycle is always a power of 2. In contrast with the cycles having an ever-doubling period introduced in Definition \ref{def:edp_cycle} (\S\ref{sub:erg}), a $Q$-cycle corresponds to some cycle of $Q_n$ whose period remains unchanged for all sufficiently large $n$, and that systematically splits into two  cycles of $Q_{n+1}$, one of which splits again, and so on as $n$ increases. From a heuristic point of view, this resembles an infinite branching process that allows one to estimate the number of short cycles of $Q_n$ for large $n$, as in \cite[\S 6]{Ber96}.

One observes in Example \ref{ex:inverse_2adic} (\S\ref{sub:parseq}) that $-1$ is a fixed point for the function $Q$, as for $T$. In fact, there are infinitely many since $-2^{k}$ and $2^{k}/3$ are fixed points for all $k \geq 0$, in addition to the trivial fixed point 0. It is conjectured  that $-1$ and $1/3$ are the only odd ones (Fixed Point Conjecture, in \cite{Ber96}). Numerically, it is easy to verify that any such point is necessarily very close, if not equal, to $-1$ or $1/3$.

In the same paper, Bernstein and Lagarias also mentioned the existence of the odd rational cycle $(-1/3, 1)$ of period 2, and conjectured that there are finitely many odd cycles for any given period $2^j$ ($3x+1$ Conjugacy Finiteness Conjecture).

Lately, I found that $(-1/5, 5/7)$ is another odd rational cycle of period 2. Indeed, 
$$Q\left( -\frac{1}{5}\right) = \overline{001}1_2 = \frac{5}{7} \quad \text{and} \quad Q\left( \frac{5}{7}\right) = \overline{0011}_2 = -\frac{1}{5}.$$
Next, I conducted a numerical verification up to period 16 on the set of rationals of the form $p/q$ where $p,q$ are odd coprime integers lower than 1000 in absolute value. Working modulo $2^{40}$ was enough to rule out the candidates that are not part of a known $Q$-cycle.

\begin{table}
\caption{Odd $Q$-cycles and corresponding $T$-cycles.}
\label{tab:cycles}
\begin{center}
\begin{tabular}{|c|c|c|}
\hline
Period & $Q$-cycle & $T$-cycles \\
\hline
\hline
1 & $(-1)$ = $\left( \overline{1}_2\right) $ & $(-1)$ \\
 & $\left( \frac{1}{3} \right) = \left(\overline{01}1_2\right)$ & $(1,2)$ \\
\hline
2 & $\left( -\frac{1}{3} , 1\right) = \left(\overline{01}_2,\overline{0}1_2\right)$ & $(0)$ and $(1,2)$ \\
 & $\left( -\frac{1}{5} , \frac{5}{7}\right) = \left(\overline{0011}_2, \overline{001}1_2\right)$ & $\left( \frac{1}{5}, \frac{4}{5}, \frac{2}{5} \right) $ and $\left( \frac{5}{7}, \frac{11}{7}, \frac{20}{7}, \frac{10}{7}\right)$ \\
\hline
\end{tabular}
\end{center}
\end{table}

In Table \ref{tab:cycles}, we list the known odd $Q$-cycles, and, for each rational element, the $T$-cycle appearing in its orbit of $T$ iterates. So far, no $Q$-cycle was found having a prime period strictly greater than 2. This leads one to think that there is none.

\begin{conjecture} {\em (Odd Cycles Conjecture)}
\label{conj:cycles}
The function $Q$ has exactly two odd fixed points, $-1$ and $\frac{1}{3}$, and two odd cycles of prime period 2,  $\left( -\frac{1}{3}, 1\right) $ and $\left( -\frac{1}{5}, \frac{5}{7}\right) $. There exists no other odd cycle, rational or not.
\end{conjecture}

\section{The $3x+1$ set}

\subsection{Euclidean embedding}
\label{sub:embed}
Overall, the automorphism $Q$ and its inverse remain somewhat mysterious. One may wish to somehow visualize their action on $\mathbb Z_2$. Recall that the space $\mathbb Z_2$ is not Euclidean and totally disconnected, which makes it difficult to represent graphically \cite{Chi96, Cuo91}. It is known to be homeomorphic to the Cantor ternary set, which has Lebesgue measure 0. We propose to apply a continuous function $M$ that sends $\mathbb Z_2$ to the real interval $[0, 2]$. The map $M$, as defined below, is very similar to the Monna\footnote{The original Monna map sends $\mathbb Z_2$ to $[0, 1]$.}  map \cite{Mon52}. 

\begin{definition} \label{def:M}
Let $M$ denote the continuous 2-Lipschitz map from $\mathbb Z_2$ to $[0, 2]$
$$ M: \sum_{k=0}^{\infty} r_{k} 2^{k} \longmapsto \sum_{k=0}^{\infty} r_k 2^{-k}, \quad \text{where $r_k=0$ or 1}.$$
\end{definition}

The action of the map $M$ may seem counter-intuitive, as it does not preserve the usual order between the rational numbers. For instance, the images of positive and negative rational numbers are deeply intertwined. But, conveniently, $M$ sends 2-adic balls with radius $2^{-k}$ onto real intervals of length $2^{1-k}$, so that the set of odd 2-adic integers is entirely mapped on the interval $[1, 2]$. We call it the ``odd'' side, whereas $[0, 1]$ may be regarded as the ``even'' side.

Let us point out that $M$ is not one-to-one, since
$$ M(1) = 1 = \sum_{k=1}^{\infty} 2^{-k} = M(-2),$$
and more generally,
$$ M\left( n+2^k\right)  = M\left( n-2^{k+1}\right)  = M(n) + 2^{-k} \quad \text{for } 0 \leq n \leq 2^{k}-1.$$
As a result, the mapping $M$ is not truly an embedding, although when restricted to $\mathbb Z_2 \setminus \mathbb Z$ it is one-to-one. Further, it is easily seen that the set $M(\mathbb Z)$ coincides with the set of dyadic numbers, namely, rationals whose denominator is a power of 2, from the interval $[0,2]$.

\begin{definition} \label{def:Qset}
Letting $X=M$ and $Y=M \circ Q$, we call ``$3x+1$" set the parametric set of the plane $\mathbb R^2$ denoted by $\Qset$ and defined by 
$$ \Qset = (X,Y)(\mathbb Z_2) = \lbrace \left( M(r), M(Q(r)) \right) : r \in \mathbb Z_2 \rbrace.$$
\end{definition}

As shown below, each point of $\Qset$ corresponds to a unique parity sequence. Somehow, the $3x+1$ set ``fully'' encodes the dynamics of the $3x+1$ map.
 
\begin{lemma} \label{lem:XY_oneone}
The function $(X,Y): \mathbb Z_2  \rightarrow [0, 2]^2$   is one-to-one and continuous with respect to the 2-adic measure on its domain.
\end{lemma}

\begin{proof}
Suppose there exist two distinct 2-adic integers $a$ and $b$ such that $X(a)=X(b)$ and $Y(a) = Y(b)$. We infer that $a$, $b$, $Q(a)$, and $Q(b)$ are all in $\mathbb Z$, and $Q(a)$ or $Q(b)$ is positive. Say $Q(a)$ is a positive integer, so it has a finite binary expansion. From the inverse formula \eqref{eq:inverse_2adic1}, it follows that $a$ is rational with denominator a power of 3 strictly greater than 1, and numerator coprime to 3. Since $a$ is a rational integer, there yields a contradiction. Hence, $(X,Y)$ is one-to-one.
 
Moreover, it is also continuous as a composition of continuous functions.
\end{proof}

Viewing $(X,Y)$ as a continuous bijection between the compact sets $\mathbb Z_2$ and $\Qset$, then its inverse is known to be continuous. Therefore $(X,Y)$ is an embedding from the parameter space $\mathbb Z_2$ into the Euclidean space $\mathbb R^2$. 

\begin{figure}
\centering
\includegraphics[scale=0.92]{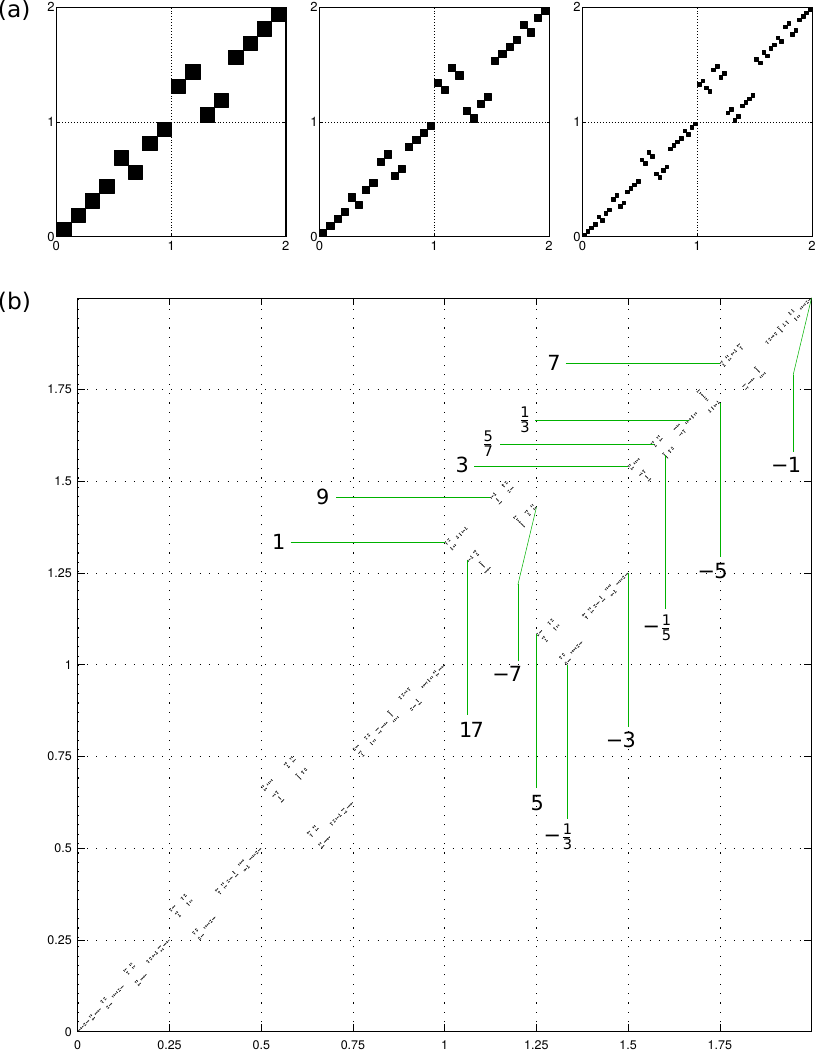}
\caption{\textbf{(a)} Coverings of $\Qset$ made of $2^k$ squares of side length $2^{1-k}$ for $k=4$, 5, and 6, from left to right. \textbf{(b)} The set $\Qset$ with (green) line segments indicating the rational points from Table \ref{tab:points}, along with their respective parameter value.}
\label{fig:QR}
\end{figure}

The fact that rationality is always preserved by the map $M$ leads to an immediate reformulation of the Periodicity Conjecture. 

\begin{conjecture} {\em (Rational Points Conjecture)}
All points in $\Qset$ have coordinates that are either both rational or both irrational.
\end{conjecture}

In Table \ref{tab:points}, we provide the coordinates of various rational points from the set $\Qset$. Among them, six points  are associated with one of the odd $Q$-cycles in Table \ref{tab:cycles} (\S\ref{sub:cycles}), whose respective parameter values are $-1$, $-1/3$, $-1/5$, $1/3$, $5/7$, and 1.

{\renewcommand{\arraystretch}{1.5}
\begin{table}
\caption{A few rational points in $\Qset$ associated with an odd rational value of the 2-adic parameter $r$ and sorted by increasing abscissa.}
\label{tab:points}
\begin{center}
\begin{tabular}{cc}

\begin{tabular}{|r|r|r|r|}
\hline
$r$ & $Q(r)$ & $X(r)$ & $Y(r)$ \\
\hline
\hline
1 & $-\frac{1}{3}$ & 1 & $\frac{4}{3}$ \\
\hline
17 & $-\frac{401}{3}$ & $\frac{17}{16}$ & $\frac{493}{384}$ \\
\hline
9 & $-\frac{6377}{3}$ & $ \frac{9}{8} $ & $ \frac{8941}{6144} $ \\
\hline
$-7$ & $ -\frac{5}{7} $ & $ \frac{5}{4} $ & $ \frac{10}{7} $ \\
\hline
5 & $ -\frac{13}{3} $ & $ \frac{5}{4} $ & $ \frac{13}{12} $ \\
\hline
$ -\frac{1}{3} $ & 1 & $ \frac{4}{3} $ & 1 \\
\hline
$-3$ & $-7$ & $ \frac{3}{2} $ & $ \frac{5}{4} $ \\
\hline
\end{tabular}

\begin{tabular}{|r|r|r|r|}
\hline
$r$ & $Q(r)$ & $X(r)$ & $Y(r)$ \\
\hline
\hline
3 & $ -\frac{23}{3} $ & $ \frac{3}{2} $ & $ \frac{37}{24} $ \\
\hline
$ \frac{5}{7} $ & $ -\frac{1}{5} $ & $ \frac{11}{7} $ & $ \frac{8}{5} $ \\
\hline
$ -\frac{1}{5} $ & $ \frac{5}{7} $ & $ \frac{8}{5} $ & $ \frac{11}{7} $ \\
\hline
$ \frac{1}{3} $ & $ \frac{1}{3} $ & $ \frac{5}{3} $ & $ \frac{5}{3} $ \\
\hline
$-5$ & $ -\frac{3}{7} $ & $ \frac{7}{4} $ & $ \frac{12}{7} $ \\
\hline
7 & $ -\frac{1595}{3} $ & $ \frac{7}{4} $ & $ \frac{2797}{1536} $ \\
\hline
$-1$ & $-1$ & 2 & 2 \\
\hline
\end{tabular}
\end{tabular}
\end{center}
\end{table}

Regardless of the 2-adic parametrization of $\Qset$, one can visualize it by taking only natural integers. This is due to the density of $\mathbb N$ in $\mathbb Z_2$, and to the fact that both functions $M$ and $Q$ are Lipschitz.  Practically, it suffices to calculate the parity vectors of length $k$ for every nonnegative integer up to $2^k$ for some $k$ reasonably large, and apply $M$ on the resulting binary expansions. We took $k=12$ in Figure \ref{fig:QR}b.

The same method  has been already used, e.g., by Hashimoto in \cite{Has07}, to represent the ``graph'' of the map\footnote{Most often, a variant of the map $T$ is considered, leading to a slower or faster dynamics.}  $T$ acting on $\mathbb Z_2$. 

A slightly different construction of the $3x+1$ set can be achieved by considering a sequence of nested sets made of finitely many squares  (Figure \ref{fig:QR}a). Although their respective areas tend to zero as the number of squares increase, we prove in Lemma \ref{lem:Qset_boxes} that they all cover $\Qset$.
\begin{lemma} \label{lem:Qset_boxes}
For all positive integers $k$, we have
$$ \Qset \subset \bigcup_{n=0}^{2^{k}-1} \left[ X(n), X(n) + 2^{1-k} \right] \times  \left[ Y_k(n), Y_k(n) + 2^{1-k} \right]$$
where $Y_k(n) = M\left( Q(n) \mod 2^k \right)$.
\end{lemma}

\begin{proof}
Let $r$ be a 2-adic integer, and let $k \geq 0$. Setting $n=r \mod 2^k$, we have
$$ X(r) \in X\left( B(n,2^{-k})\right)  = \left[ X(n), X(n) + 2^{1-k} \right]$$
and
$$ Y(r) \in Y\left( B(n,2^{-k})\right)  = M\left( B(Q(n),2^{-k})\right) 
= \left[ Y_k(n), Y_k(n) + 2^{1-k} \right].$$
The inclusion claimed in Lemma \ref{lem:Qset_boxes} follows. 
\end{proof}

From this result, we infer that
\begin{equation}
\Qset \subset \bigcap_{k=0}^{\infty} \bigcup_{n=0}^{2^{k}-1} \left[ X(n), X(n) + 2^{1-k} \right] \times  \left[ Y_k(n), Y_k(n) + 2^{1-k} \right].
\end{equation}
In fact, the equality holds, as both sets have exactly the same number of intersections with every line parallel to the $y$-axis.

Another corollary of Lemmas \ref{lem:XY_oneone} and \ref{lem:Qset_boxes} is the presence of infinitely many discontinuities in $\Qset$ for the Euclidean metric, at each point whose abscissa is dyadic, except the extremal points (0,0) and (2,2).
 
 Let us observe in Figure \ref{fig:QR}b that it has a rather symmetric aspect with respect to the diagonal $\Delta = \lbrace(x,x) : 0 \leq x \leq 2\rbrace$. This is mainly due to the congruence \eqref{eq:Q2} in Corollary \ref{cor:Q_iter}. An underlying question would be to determine how much does the function $Q$ differ from its inverse, about which little is known. From Figure \ref{fig:QR}a, it is clear that the symmetry is broken only at a rather small scale. It turns out that few points are effectively symmetric. Despite the non-injectivity of the map $M$, it is most likely that only those points whose parameter values are part of a $Q$-cycle of period at most 2, are symmetric. We obtain thereby  a symmetric subset that is expected to contain exactly six points in the  upper right quarter of the set $\Qset$, two of them being on $\Delta$ (see Conjecture \ref{conj:cycles} in \S\ref{sub:cycles}, and Table \ref{tab:points}).

One may further notice a number of affine self-similarities, some of which are made explicit in the next section (\S\ref{sub:autosim}).

\subsection{Self-similarity}
\label{sub:autosim}

As a result of the functional equations \eqref{eq:Q_feq_2x} and \eqref{eq:Q_feq_2xp1} satisfied by $Q$, it is possible to delimit regions of $\Qset$ that are identical through an affine transformation. To this aim, we first introduce two infinite families of real intervals, which realize a covering of the half-open interval $\left[  0, 2 \right) $.
\begin{definition} \label{def:IJ}
For every integer $k \geq 2$, let
$$ \alpha_k = -1 - (-2)^{k-2} \mod 2^{k}, \quad m_k = 2^{k-2} - 1, \quad n_k = 3 \cdot 2^{k-2} -1,$$
so that $\alpha_k = m_k$ if $k$ is odd, and $\alpha_k = n_k$ otherwise. Then define the real intervals
$$ I_k = \left[  M(m_k), M(n_k) \right]  = \left[  2 - 2^{3-k}, 2 - 3\cdot 2^{1-k}\right] $$
and
$$ J_k = \left[  M(n_k), M(m_{k+1}) \right]  = \left[  2 - 3\cdot 2^{1-k}, 2 - 2^{2-k}\right]$$
of length $2^{1-k}$. The mapping $M$ sends the 2-adic ball $B(\alpha_k,2^{-k})$ onto $I_k$ or $J_k$ alternatively, according to the parity of $k$.
\end{definition}

The next lemma, along with Corollary \ref{cor:ball_to_interval}, will prove useful to delimit in the $3x+1$ set all parts corresponding to parametric values in the same congruence class as $m_k$ or $n_k$ modulo $2^k$.
\begin{lemma}
\label{lem:mk_nk}

The integers $(m_k)_{k\geq2}$ and $(n_k)_{k\geq2}$ have the properties
\begin{align}
 &Q(m_k) \equiv m_k \pmod{2^k}  &\text{and} \quad &Q(n_k) \equiv n_k \pmod{2^k} & \text{for $k$ even,} \label{eq:mnk_even} \\
 &Q(m_k) \equiv n_k \pmod{2^k}  &\text{and} \quad &Q(n_k) \equiv m_k \pmod{2^k} & \text{for $k$ odd.} \label{eq:mnk_odd}
\end{align}
\end{lemma}

\begin{proof}The function $Q$ induces a permutation on $\mathbb{Z}/2^k\mathbb{Z}$. Thus, we can reason on $Q^{-1}$ instead of $Q$.

Let us write, first, the binary representations 
$$ m_k = \underbrace{0011\ldots1_{2}}_{k} \quad \text{and} \quad n_k = \underbrace{1011\ldots1_{2}}_{k}.$$ 
Applying formula \eqref{eq:inverse_2adic2} from Corollary \ref{cor:inverse_2adic}, we get
\begin{align*}
Q^{-1}(m_k) & \equiv -1 - 2^{k-2}3^{2-k} - 2^{k-1}3^{2-k} \pmod{2^k}\\
& \equiv -1 - 2^{k-2}3^{1-k} \pmod{2^k} \\
& \equiv -1 + (-2)^{k-2} \pmod{2^k}
\end{align*}
since $\frac13 \equiv -1 \pmod{4}$, and similarly,
\begin{align*}
Q^{-1}(n_k) & \equiv -1 - 2^{k-2}3^{2-k} \pmod{2^k} \\
& \equiv -1 - (-2)^{k-2} \pmod{2^k}.
\end{align*}
The properties \eqref{eq:mnk_even} and \eqref{eq:mnk_odd} follow by considering the parity of $k$.
\end{proof}

\begin{corollary}\label{cor:ball_to_interval}
For every $k \geq 2$,
$$ Y\left( B(\alpha_k,2^{-k})\right) = J_k \quad \text{and} \quad Y\left( B(2\alpha_k+1,2^{-k-1})\right) = I_{k+1},$$
where $Y=M \circ Q$ and $B(\alpha_k,2^{-k})$ is the closed ball with center $\alpha_k$ and radius $2^{-k}$ in $\mathbb Z_{2}$ (see Definitions  \ref{def:ball} in \S\ref{sub:erg}, and \ref{def:M} in \S\ref{sub:embed}).
\end{corollary}

\begin{proof}
From Lemma \ref{lem:mk_nk}, we have
$$ Q(\alpha_k) \equiv n_k \pmod{2^k} \quad \text{and} \quad Q(2 \alpha_k + 1) \equiv m_{k+1} \pmod{2^{k+1}},$$
whatever the parity of $k$. Since $Q$ is an isometry, it yields
$$Q(B(\alpha_k,2^{-k})) = B(n_k,2^{-k}) \quad \text{and} \quad  Q(B(2 \alpha_k + 1,2^{-k-1})) = B(m_{k+1},2^{-k-1}).$$
Hence the result.
\end{proof}

We now establish the existence of infinitely many affine relationships within $\Qset$ and give their analytic expressions.
\begin{theorem} \label{th:autosim}
The set $\Qset = (X,Y)(\mathbb Z_2)$ admits the self-affine relationships
\begin{equation} \label{eq:trans_even}
( X, Y )(2r) = \frac12 ( X, Y)(r) 
\end{equation}
for $r \in \mathbb Z_{2}$, and
\begin{equation} \label{eq:trans_odd}
( X, Y )(2r+1) = \frac12 ( X, Y)(r) + ( 1 , 1-2^{-k} )
\end{equation}
for $r \equiv \alpha_k \pmod{2^k}$ and $k \geq 2$, where $X=M$ and $Y=M \circ Q$ as in Definition \ref{def:M} (\S\ref{sub:embed}).
\end{theorem}
\begin{proof}[Proof of \eqref{eq:trans_even}]
For all 2-adic integers $r$, $$M(2r) = \frac12 M(r)$$ and, by the functional equation \eqref{eq:Q_feq_2x}, $$M(Q(2r)) = M(2Q(r)) = \frac12 M(Q(r)).$$
\end{proof}

\begin{proof}[Proof of \eqref{eq:trans_odd}]
Let $r \equiv \alpha_k \pmod{2^k}$. It is easily seen that $$M(2r+1) = 1 + \frac12 M(r).$$
Now, recall the functional equation \eqref{eq:Q_feq_2xp1}:
$$Q(2r+1) = 2Q(r) + 1 - 2^k.$$
Lemma \ref{lem:mk_nk} gives $Q(r) \equiv n_k \pmod{2^k}$, yielding the 2-adic expansion
$$ Q(r) = 1 + 2 + 2^2 + \ldots + 2^{k-3} + 2^{k-1} + \ldots.$$
Hence, we obtain
$$M(Q(2r+1)) = 1 + \frac12 M(Q(r)) - 2^{-k}.$$
\end{proof}

{\renewcommand{\arraystretch}{1.5}
\begin{table} 
\caption{Pairs of boxes covering two parts of $\Qset$ that coincide modulo an affine transformation with scaling factor $\frac{1}{2}$. Note that $k \geq 2$.}
\label{tab:squares}
\begin{center}
\begin{tabular}{|c|c|c|c|}
\hline
Parity &  Box 1  &  Box 2 & Affine transformation  \\
 \hline
 \hline
- & $[0, 2]^2$ & $[0, 1]^2$ & $(x,y) \longmapsto \left( \frac{x}{2}, \frac{y}{2} \right)$ \\
$k$ even & $J_{k}^{2}$ & $J_{k+1} \times I_{k+1}$ & $(x,y) \longmapsto \left( 1+\frac{x}{2}, 1+\frac{y}{2}-2^{-k} \right)$ \\
$k$ odd & $I_{k} \times J_{k}$ & $I_{k+1}^{2}$ & $(x,y) \longmapsto \left( 1+\frac{x}{2}, 1+\frac{y}{2}-2^{-k} \right)$ \\
 \hline
\end{tabular}
\end{center}
\end{table}
}

\begin{figure}
\centering
\includegraphics[scale=0.6]{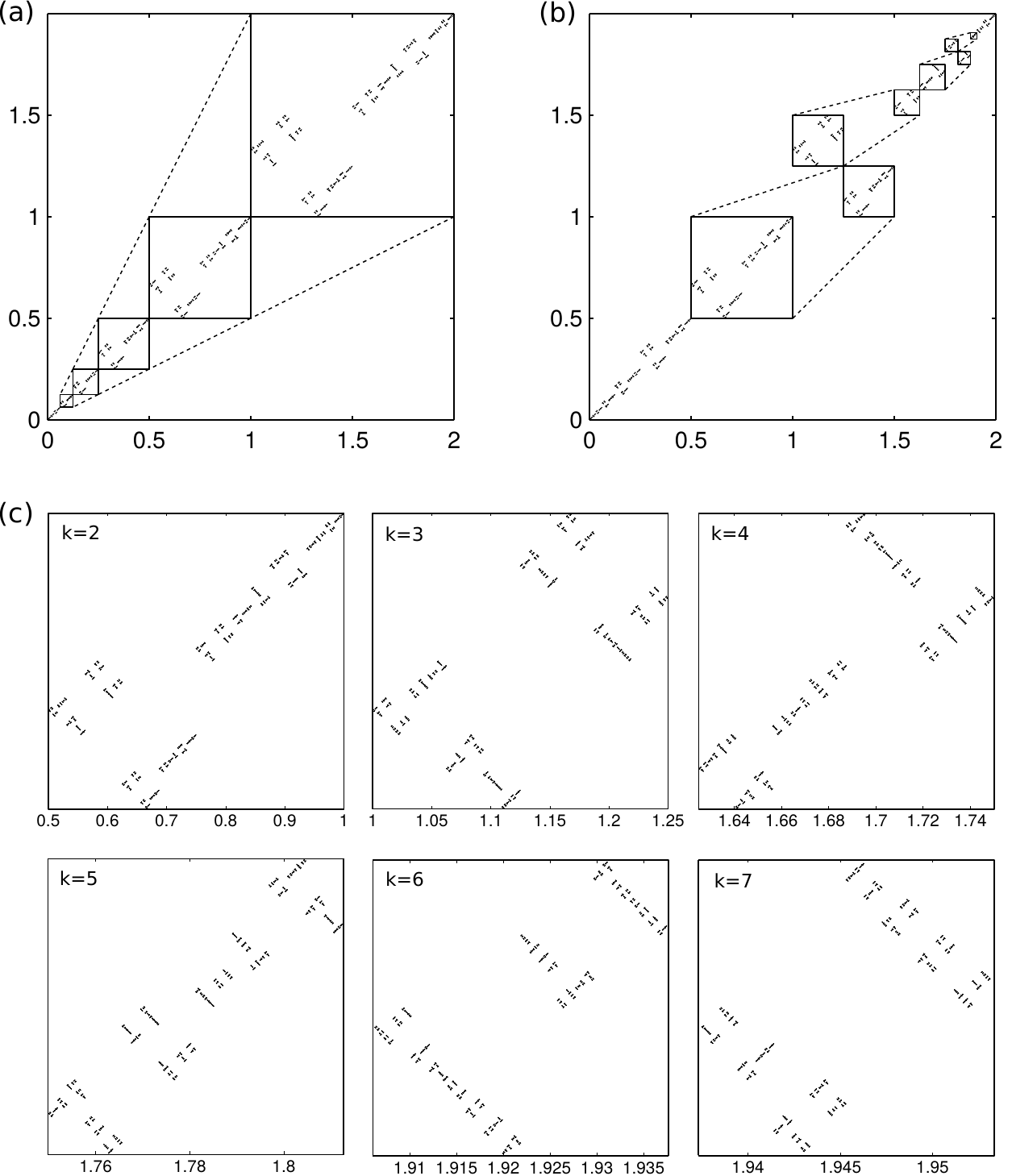}
\caption{\textbf{(a-b)} Identical parts of $\Qset$ through the affine transformations \eqref{eq:trans_even} and \eqref{eq:trans_odd} in (a) and (b) respectively. \textbf{(c)} Enlarged parts of $\Qset$ delimited by some of the boxes in (b), namely, $J_{k}^{2}$ for $k$ even and $I_{k} \times J_{k}$ for $k$ odd, with $2 \leq k \leq 7$.}
\label{fig:QR_squares}
\end{figure}

It follows from Theorem \ref{th:autosim} together with Corollary \ref{cor:ball_to_interval} that some parts of the $3x+1$ set described in Table \ref{tab:squares} are identical through the affine transformations represented on Figures \ref{fig:QR_squares}a and \ref{fig:QR_squares}b. The  first line of Table \ref{tab:squares} is linked to the autoconjugacy \eqref{eq:Q_feq_2x}. The fact that the boxes $[0, 2]^2$ and $[0, 1]^2$ are nested leads to an infinite descent on the ``even" side of $\Qset$ towards the origin. On the other hand, the equation \eqref{eq:trans_odd} applies only for odd values of the parameter $r$, except when $k=2$. From the corresponding pairs of boxes in Table \ref{tab:squares}, we obtain a covering of the ``odd'' side of $\Qset$, which takes the form in Figure \ref{fig:QR_squares}b of an infinite ``cascade'' along the diagonal $\Delta$.

Other functional equations like \eqref{eq:Qinv_feq_2xp1} and \eqref{eq:Q_feq_3xp1} do not imply self-similarity on our plane representation in Figure \ref{fig:QR}b, due to the value 3 not being a power of 2.

On Figure \ref{fig:QR_squares}c, we show enlarged parts of $\Qset$, mainly from the ``odd" side, that are delimited by squares of side-length $2^{1-k}$ for $2\leq k \leq 7$. Putting together all the affine similarities, it yields that the content of the first square $J_{2}^{2}$ is made entirely of small copies of $J_{k}^{2}$ for even $k \geq 4$, and  $I_{k} \times J_{k}$ for odd $k \geq 3$, plus an extra point at $\left( \frac{1}{2}, \frac{1}{2}  \right) $. It is a puzzling question whether every square on Figure \ref{fig:QR_squares}c is also made of small pieces taken elsewhere within the odd side of $\Qset$ and outside $J_{3} \times I_{3}$.

Nonetheless, one observes some relative diversity of patterns. Unlike for the Cantor ternary set, there seems to be no simple geometric scheme able to reproduce $\Qset$, in the sense that more and more calculations are required for refining the shape of each pattern.

Finally, the unveiling of self-similarity at all scales raises the question of the Hausdorff dimension, to which the following theorem answers without much difficulty.

\begin{theorem}
The set $\Qset$ has Hausdorff dimension 1.
\end{theorem}

\begin{proof}
First, observe that the Hausdorff dimension of $\Qset$ is at least 1, as it contains at least one point for each abscissa taken in the interval $[0,2]$.

For all $k \geq 0$, we obtain from Lemma \ref{lem:Qset_boxes} (\S\ref{sub:embed}) a covering of $\Qset$ made of $\nu_k = 2^k$ boxes of side-length $l_k = 2^{1-k}$. The number of boxes is minimal because the number of intervals of length $l_k$ required to cover $[0,2]$ is at least $2^k$.

Therefore the ``box-counting" dimension of $\Qset$ is equal to
$$\lim_{k \rightarrow \infty}  \frac{\log \nu_k}{\log\left( \frac{1}{l_k}\right) } = \lim_{k \rightarrow \infty}  \frac{k}{k-1} = 1,$$
which is an upper bound of its Hausdorff dimension. 
\end{proof}

\section*{Acknowledgements}
I am indebted to the anonymous referee for his valuable comments that helped to improve the exposition of the $3x+1$ set.

\textsc{Institut de Physique du Globe de Paris, France}

{\it E-mail: }{\tt rozier@ipgp.fr}


\begin{thebibliography}{99}
\bibitem{Aki04}
E. Akin,
Why is the 3x+1 problem hard?,
{\it Contemp. Math.} {\bf 356} (2004), 1--20.

\bibitem{Ana06}
V. Anashin, 
Ergodic transformations in the space of $p$-adic integers,
in {\it p-Adic Mathematical Physic}, AIP Conference Proceedings, Vol. 826, Belgrade, 2006, 3--24. Available at 
\url{https://arxiv.org/abs/math/0602083} .

\bibitem{Ana14}
V. Anashin, A. Khrennikov, E. Yurova,
Ergodicity criteria for non-expanding transformations of 2-adic spheres,
{\it Discrete Contin. Dyn. Syst.} {\bf 34} (2014), 367--377.

\bibitem{And00}
P. Andaloro, 
On total stopping times under 3x+1 iteration,
{\it Fibonacci Quart.} {\bf 38} (2000), 73--78.

\bibitem{Ber94}
D. J. Bernstein,
A non-iterative 2-adic statement of the $3N+ 1$ conjecture,
{\it Proc. Amer. Math. Soc.} {\bf 121} (1994), 405--408.

\bibitem{Ber96}
D. J. Bernstein, J. C. Lagarias,
The $3x+ 1$ conjugacy map,
{\it Can. J. Math.} {\bf 48} (1996), 1154--1169.

\bibitem{Boh78}
C. B\"{o}hm, G. Sontacchi,
On the existence of cycles of given length in integer sequences like $x_{n+1}= x_n/2$ if $x_n$ even, and $x_{n+1}= 3x_n+1$ otherwise,
{\it Atti Accad. Naz. Lincei Sci. Fis. Mat. Natur.} {\bf 64} (1978), 260--264.

\bibitem{Chi96}
D. V. Chistyakov, 
Fractal geometry for images of continuous embeddings of $p$-adic numbers and solenoids into Euclidean spaces,
{\it Theoret. Math. Phys.} {\bf 109} (1996), 1495--1507.

\bibitem{Coq83}
J. Coquet,
A summation formula related to the binary digits,
{\it Invent. Math.} {\bf 73} (1983), 107--115.

\bibitem{Cra78}
R. E. Crandall,
On the ``$3x+1$'' problem, 
{\it Math. Comp.} {\bf 32} (1978), 1281--1292.

\bibitem{Cuo91}
A. A. Cuoco, 
Visualizing the p-adic integers,
{\it Amer. Math. Monthly} {\bf 98} (1991), 355--364.

\bibitem{Edg12}
A. Edgington, 
The autoconjugacy of a generalized Collatz map,
preprint available at \url{https://arxiv.org/abs/1206.0553} (2012), 1--6.

\bibitem{Eve77}
C. J. Everett,
Iteration of the number theoretic function $f(2n)=n, f(2n+1)=3n+2$,
{\it Adv. Math.} {\bf 25} (1977), 42--45.

\bibitem{Gar85}
L. E. Garner, On heights in the Collatz 3n+ 1 problem, 
{\it Discrete Math.} {\bf 55} (1985), 57--64.

\bibitem{Guy}
R. K. Guy,
{\em Unsolved Problems in Number Theory}, Third Edition, Springer, 2004.

\bibitem{Has07}
Y. Hashimoto,
A fractal set associated with the Collatz problem,
{\it Bull. of Aichi Univ. of Education (Natural science)} {\bf 56} (2007), 1--6. Available at 
\url{https://www.researchgate.net/profile/Yukihiro_Hashimoto} .

\bibitem{Lag85}
J. C. Lagarias,
The $3x+1$ problem and its generalizations,
{\it Amer. Math. Monthly} {\bf 92} (1985), 3--23.

\bibitem{Lag90}
J. C. Lagarias,
The set of rational cycles for the $3x+ 1$ problem,
{\it Acta Arith.} {\bf 56} (1990), 33--53.

\bibitem{Lag10}
J. C. Lagarias (editor),
{\em The Ultimate Challenge: The 3x+1 Problem}, Amer. Math. Soc., 2010.

\bibitem{Lop09}
J. L\'opez, P. Stoll,
The 3x+1 conjugacy map over a Sturmian word,
{\it Integers} {\bf 9} (2009), 141--162.

\bibitem{Mat84}
K. R. Matthews, A. Watts,
A generalization of Hasse's generalization of the Syracuse algorithm, 
{\it Acta Arith.} {\bf 43} (1984), 167--175.

\bibitem{Mol78}
H. M\"oller,
\"Uber Hasse's Verallgemeinerung des Syracuse-Algorithmus (Kakutani's Problem),
{\it Acta Arith.} {\bf 34} (1978), 219--226.

\bibitem{Mon04}
K. G. Monks,  J. Yazinski,
The autoconjugacy of the $3x+ 1$ function,
{\it Discrete Math.} {\bf 275} (2004), 219--236.

\bibitem{Mon13}
K. Monks, K. G. Monks, K. M. Monks, M. Monks,
Strongly sufficient sets and the distribution of arithmetic sequences in the $3x+ 1$ graph,
{\it Discrete Math.} {\bf 313} (2013), 468--489.

\bibitem{Mon52}
A. F. Monna,
Sur une transformation simple des nombres P-adiques en nombres r\'eels,
{\it Indag. Math.} {\bf 14} (1952), 1--9.

\bibitem{Mul91}
H. M\"uller, 
Das `3n+1' Problem,
{\it Mitt. Math. Ges. Hamburg} {\bf 12} (1991), 231--251.

\bibitem{Roz17}
O. Rozier,
The $3 x+ 1$ problem: a lower bound hypothesis, 
{\it Funct. Approx. Comment. Math.} {\bf 56} (2017), 7--23.

\bibitem{Sim16}
R. Simonetto,
{\it Conjecture de Syracuse : Avanc\'ees In\'edites},
2016. Available at \url{https://mathsyracuse.wordpress.com} .

\bibitem{Ter76}
R. Terras,
A stopping time problem on the positive integers,
{\it Acta Arith.} {\bf 30} (1976), 241--252.

\end{thebibliography}
\end{document}